\documentclass[12pt,psfig]{article}
\usepackage{float}
\usepackage{color}
\usepackage{subfigure}
\usepackage{amsmath}
\usepackage{amsthm}
\usepackage{amsfonts}
\usepackage{amssymb,latexsym}
\usepackage[all]{xy}
\usepackage{graphicx}
\usepackage[mathscr]{eucal}
\usepackage{verbatim}
\usepackage{hyperref}

\theoremstyle{plain}

    \newtheorem{thm}{Theorem}[section]
       
       \newtheorem{lem}{Lemma}[section]
       \newtheorem{defn}{Definition}[section]
       \newtheorem{rem}{Remark}[section]

\numberwithin{equation}{section}

\usepackage{graphicx}
\usepackage{pict2e}

\begin{document}
\title{Random attractors for a stochastic nonlocal delayed reaction-diffusion equation on a semi-infinite interval}

\author{Wenjie Hu$^{1,2}$,  Quanxin Zhu$^{1,3}$, Tom\'{a}s
Caraballo$^{4}$\footnote{Corresponding author.  E-mail address: caraball@us.es (Tom\'{a}s
Caraballo).}
\\
\small  1. The MOE-LCSM, School of Mathematics and Statistics,  Hunan Normal University,\\
\small Changsha, Hunan 410081, China\\
\small  2. Journal House, Hunan Normal University, Changsha, Hunan 410081, China\\
\small  3. the Key Laboratory of Control and Optimization of Complex Systems, College of Hunan Province,\\
\small Hunan Normal University, Changsha 410081, China.\\
\small 4 Dpto. Ecuaciones Diferenciales y An\'{a}lisis Num\'{e}rico, Facultad de Matem\'{a}ticas,\\
\small  Universidad de Sevilla, c/ Tarfia s/n, 41012-Sevilla, Spain
}

\date {}
\maketitle

\begin{abstract}
The aim of this paper is to prove  the existence and qualitative property of   random attractors  for a stochastic nonlocal delayed reaction-diffusion equation (SNDRDE) on a semi-infinite interval with a Dirichlet boundary condition on the finite end. This equation models the spatial-temporal evolution of the mature individuals for a two-stage species whose juvenile and adults both diffuse that lives on a semi-infinite domain and subject to random perturbations. By transforming the SNDRDE into a random evolution equation with delay, by means of a stationary conjugate transformation, we first establish the global existence and uniqueness of solutions to the equation, after which we show the solutions generate a random dynamical system. Then, we deduce uniform a priori estimates of the solutions  and show the existence of bounded random absorbing sets. Subsequently, we prove the pullback asymptotic compactness of the random dynamical system generated by the SNDRDE with respect to the compact open topology, and hence obtain the existence of  random attractors. At last, it is proved that the random attractor is an exponentially attracting stationary solution under appropriate conditions.
\end{abstract}


\section{Introduction}
When modeling the growth of mature population  of a two-stage species (juvenile and adult, with a fixed maturation time $\tau$) whose mature individuals and immature individuals both diffuse,  one faces the following delayed reaction-diffusion equation with spatial non-locality derived in \cite{SWZ}
\begin{equation}\label{1}
 \displaystyle\frac{\partial u(t,x)}{\partial
 t}=\displaystyle  \Delta u(t,x)-\mu u(t,x)+\varepsilon\int_{\mathcal{O}} \Gamma(\alpha,x,y)f(u(t-\tau,y))\mathrm{d}y,\ (t,x)\in(0,\infty)\times D.
\end{equation}
Here, $\mathcal{O}\subseteq \mathbb{R}^N$ is the spatial domain, $u(t,x)$ stands for the total mature population at location  $x$ and time $t$. The positive constant $\mu$ represents the death rate of the mature population. $\varepsilon$ and the immature mobility constant $\alpha$ are defined by  $\varepsilon
=\displaystyle e^{-\int^r_0{d_I(a)\mathrm{d}a}}$ and $\alpha=\int^r_0{D_I(a)\mathrm{d}a}$, where $d_I(a)$, $D_I(a)$, $a\in [0, \tau]$  are the age dependent death rate and diffusion rate of the immature population of the species respectively. The diffuse kernel $\Gamma(\alpha,x,y)$ is obtained by integrating along the characteristic  based on the general model in \cite{JMO}, representing the probability  of the new born individuals located at $y$ that can survive to be matured and moved to location $x$. Generally speaking, explicit forms of $\Gamma(\alpha,x,y)$ can only be obtained for some special cases, see \cite{LSZZ}. When the spatial domain $\mathcal{O}$ is bounded with a Dirichlet boundary value condition (DBVC), existence, uniqueness and attractiveness of the positive steady state and threshold dynamics are important and have been explored by Yi and Zou  \cite{YXA}. When a Neumann boundary value condition (NBVC) is imposed,   Zhao \cite{ZZ} established the global attractiveness of the positive steady state of ($\ref{1}$) by adopting a fluctuation method.

In the real world, there are also species whose individuals live in the whole space $\mathbb{R}$  or the semi-infinite domain $\mathbb{R}_+$. In the situation $\mathcal{O}=\mathbb{R}$,  the lack of compactness of the infinite domains and the complexity of nonlocal delayed term cause the global dynamics analysis of \eqref{1} becomes quite difficult and hence, the existing works  mainly focus on the traveling wave solutions (see, for instance, \cite{SWZ,WZ,YZA}). To circumvent this difficulty, Yi et al. \cite{YCWT} made a first attempt to describe the global dynamics of model (\ref{1}) by adopting the compact open topology combined with delicate analysis of the asymptotic properties of the nonlocal term and the diffusion operator. For the species whose individuals live in a semi-infinite domain, the kernel function and the spatial domain are neither symmetric nor compact, implying the problem becomes more challenging. Recently,  Yi and Zou \cite{YZD} derived the kernel
\begin{equation}\label{2}
\Gamma(\alpha,x,y)=\frac{1}{\sqrt{4\pi \alpha}}e^{-\frac{(x-y)^2}{4\alpha}}-\frac{1}{\sqrt{4\pi \alpha}}e^{-\frac{(x+y)^2}{4\alpha}},
\end{equation}
in the scenario $\mathcal{O}=\mathbb{R}_+=[0, \infty)$ with the homogeneous DBVC at the finite end, and investigated global dynamics of  Eq. \eqref{1}. Hu and Duan \cite{6} discussed the global asymptotic behavior of solutions to (\ref{1}) on a half plane   with a NBVC.

However, the  evolution of the mature population is inevitably affected by random perturbations,  including the noise generated by the internal self excitation of the system and the random interference of the external environment. Consequently, for the species living on $\mathcal{O}=\mathbb{R}_+$ that are perturbed by some random effects, a more accurate mathematical model should be the following stochastic nonlocal delayed reaction-diffusion equation (SNDRDE),
\begin{equation}\label{3}
 \displaystyle\frac{\partial u}{\partial
 t}(t,x)= \Delta u(t,x)-\mu u(t,x)+\varepsilon\int_{\mathcal{O}} \Gamma(\alpha,x,y)f(u(t-\tau,y))\mathrm{d}y+\sum_{j=1}^{m} g_{j}(x)\frac{d w_{j}}{dt},\ (t,x)\in(0,\infty)\times \mathbb{R}_+,
\end{equation}
which is obtained by adding an additive noise $\sum_{j=1}^{m} g_{j}(x)\frac{\mathrm{d} w_{j}}{\mathrm{d}t}$ to \eqref{1}. Here, $ \Gamma(\alpha,x,y)$ is defined by \eqref{2}, $\{g_j(x)\}_{j=1}^{m}$ are twice continuously differentiable on $(0,\infty)$, standing  for the intensity and the shape of noise, $\{w_j\}_{j=1}^{m}$ are mutually independent two-sided real-valued Wiener process on an appropriate probability space to be specified later.

In order to obtain the  global complex dynamics and nonlocal analysis of the qualitative properties of random dynamical systems,  Crauel and Flandoli proposed the concept of random attractors for infinite dimensional random system in \cite{9,10,FS}, by generalizing the theory of global attractors of infinite dimensional dissipative systems. Since then,  the existence, finite dimensionality and structure of random attractors for various stochastic nonlinear evolution equations or stochastic functional differential equations have been  extensively and intensively investigated  by adopting the framework in \cite{9,10}. For example, for the stochastic reaction-diffusion equation without time delay, Caraballo et al. \cite{11}, Gao et al. \cite{13} and Li and Guo \cite{12} explored the existence of global attractors on bounded domains. For the stochastic delayed reaction-diffusion equation on bounded domains, the existence of random attractors and  their structure have been studied in \cite{17,16,18,LG,25}.

In our recent works \cite{HZ} and \cite{HZ20}, we have obtained the existence, uniqueness and stability of solutions to \eqref{1} as well as the existence of random attractors when the domain $\mathcal{O}$ is bounded with a DBVC. Therefore, similar questions arise naturally, i.e. under what conditions does  \eqref{3} admit a unique global solution? Under what conditions does  \eqref{3} generate a random dynamical system? Under what conditions does  \eqref{3} possess random attractors? Moreover, under what conditions is  the attractor of \eqref{3}   a random fixed point? In the recent works \cite{24}, \cite{15}  and \cite{14}, the authors obtained the existence of global attractors for stochastic reaction-diffusion equations on unbounded domains. The unboundedness of the domain  causes the Sobolev embeddings  are no longer compact and the asymptotic compactness of solutions cannot be obtained by a standard method. Therefore, in order to overcome the difficulty caused by the unboundedness of the domain, \cite{24} established uniform estimates on the far-field values of solutions. Nevertheless, it follows from \eqref{2} that the kernel $\Gamma(\alpha,x,y)$ is asymmetric and the domain is non-symmetric and noncompact which together with the time delay, imply that the analysis of the long time behavior of solutions to \eqref{3} on the semi-infinite interval $\mathbb{R}_+=[0, \infty)$ is more difficult. This motivates us to establish a new method to analyze the asymptotic behavior of the following stochastic initial boundary value problem
\begin{equation}\label{4}
\left\{\begin{array}{l}
\frac{\partial u}{\partial t}(t,x)= \Delta u(t,x)-\mu u(t,x)+\varepsilon\int_{\mathbb{R}_{+}} \Gamma(\alpha,x,y)f(u(t-\tau,y))\mathrm{d}y+\sum_{j=1}^{m}g_{j}(x)\frac{d w_{j}}{dt},\\
u(t, 0)=0, \quad t>0, \\
u(t, x)=\phi(t, x), \quad(t, x) \in[-\tau, 0] \times \mathbb{R}_{+}.
\end{array}\right.
\end{equation}

In the case of deterministic equations, to obtain the  global dynamics of \eqref{1}, Yi and Zou \cite{YZD} established a priori estimates for nontrivial solutions by exploring the asymptotic properties of the nonlocal delayed effect and the diffusion operator. This also has been adopted by Hu et al \cite{6, HDZ} to explore the global dynamics of  some nonlocal delayed differential equations on different half spaces with various boundary conditions. Introducing random factors cause the analysis of the  asymptotic behavior of \eqref{4} be quite different from the deterministic case since the existence and uniqueness of global solution to SNDRDE \eqref{4}, and whether it generates a random dynamic system are not so natural as in the deterministic case. Moreover, the framework to deal with random attractors is also quite different from that of the deterministic case.   In this paper, we carry out a first attempt to extend the method of exploring the  asymptotic properties of the deterministic nonlocal delayed effect and the diffusion operator to the random case,  and prove the existence and qualitative property of  random attractors for the SNDRDE \eqref{4}  on the unbounded domain $\mathbb{R}_{+}$. Unlike the previous works \cite{17,16,18,LG,25}, where the phase space is a Hilbert space, we need to work here with a Banach space as natural phase space. Due to the lack of an inner product, we prove the existence of a global solution and obtain uniform a priori estimates of the solution by using the semigroup approach together with a careful analysis of the diffusion operator instead of taking inner product. Moreover, to overcome the difficulty caused by the noncompactness of the spatial domain, we also adopt the compact open topology to describe the  asymptotic behavior. It is clear that our method can be used for a variety of other equations on half spaces, as it was done for the deterministic case \cite{HDZ}.

The remaining part of this paper is structured as follows. In Section 2, we recall some basic results from the theory of random dynamical systems and random attractors as well as some notation and preliminary lemmas needed for the proof of our main results. In Section 3, by means of the Ornstein-Uhlenbeck (O-U) process, we first transform  SNDRDE \eqref{4} into a random partial differential equation with delay, and we then show that SNDRDE \eqref{4} has  global solutions by the Banach fixed point theorem together with the properties of the semigroup generated by the linear part of  \eqref{4}. Furthermore,  we show that solutions to \eqref{4}  generate a random dynamical system.  To prove the existence of random attractor for  SNDRDE \eqref{4}, we first establish  uniform a priori estimates of the solutions in Section 4, and we then show the asymptotic compactness of random dynamical systems generated by \eqref{2}  with respect to the compact open topology, implying the existence of random attractors by the results in  \cite{9,10}. In Section 5, we derive sufficient conditions  ensuring   the random attractor becomes an exponentially attracting stationary solution. At last, we summarize the paper by pointing out some potential directions deserving further  research.
\section{Preliminaries}
We first recall some notation to be used throughout this paper, and then we introduce the theory of random dynamical systems as well as random attractors. Denote by $BUC\left(\mathbb{R}_{+}, \mathbb{R}\right)$ the set of all bounded and uniformly continuous functions from $\mathbb{R}_{+}$ to $\mathbb{R}$. Denote by $\mathcal{C}=C([-\tau, 0], X)$ the set of all  continuously functions from $[-\tau, 0]$ to $X$  equipped with the usual supremum norm  $\|\varphi\|_{\mathcal{C}} =\sup\{\| \varphi(\xi)\|_{X} :\xi \in [-\tau,0]\}$ for any $\varphi\in \mathcal{C}$. For any real interval $J\subseteq\mathbb{R}$, set $J+[-\tau,0]=\{t+\xi:t\in I \ \mathrm{and}  \ \xi\in[-\tau,0] \}$. For any $ u:(J+[-\tau,0])\rightarrow X $ and $ t\in J$, we define $ u_{t}(\cdot)\in \mathcal{C }$ by $ u_{t}(\xi)=u(t+\xi)$ for all $\xi\in[-\tau,0]$.

In the sequel, we introduce the concept of random attractor and random dynamical system following \cite{AL} and  \cite{9,10,FS}.
\begin{defn}\label{defn1}
Let $\left\{\theta_{t}: \Omega \rightarrow \Omega, t \in \mathbb{R}\right\}$ be a family of measure preserving transformations such that $(t, \omega) \mapsto \theta_{t} \omega$ is measurable and $\theta_{0}=\mathrm{id}$, $\theta_{t+s}=\theta_{t} \theta_{s},$ for all $s, t \in \mathbb{R}$. The flow $\theta_{t}$ together with the probability space $\left(\Omega, \mathcal{F}, P,\left(\theta_{t}\right)_{t \in \mathbb{R}}\right)$ is called a metric dynamical system.
\end{defn}
 For a given complete separable metric space $(X, d)$,  denote by $\mathcal{B}(X)$ the  Borel-algebra of open subsets in $X$.
\begin{defn}\label{defn2}
A mapping $\Phi: \mathbb{R}^{+} \times \Omega \times X   \rightarrow X$ is said to be a random dynamical system (RDS) on a complete separable metric space  $(X,d)$ with Borel  $\sigma$-algebra  $\mathcal{B}(X)$ over the metric dynamical system $\left(\Omega, \mathcal{F}, P,\left(\theta_{t}\right)_{t \in \mathbb{R}}\right)$ if \\
(i) $\Phi(\cdot, \cdot, \cdot): \mathbb{R}^{+} \times \Omega \times X   \rightarrow X$ is $(\mathcal{B}(\mathbb{R}^{+})\times \mathcal{F}\times\mathcal{B}(X), \mathcal{B}(X))$-measurable;\\
(ii) $\Phi(0, \omega,\cdot)$ is the identity on  $X$ for $P$-a.e. $\omega \in \Omega$;\\
(iii) $\Phi(t+s, \omega,\cdot)=\Phi(t, \theta_{s} \omega,\cdot) \circ \Phi(s, \omega,\cdot),   \text { for all } t, s \in \mathbb{R}^{+}$ and $P$-a.e. $\omega \in \Omega$.\\
A RDS $\Phi$ is continuous or differentiable if $\Phi(t, \omega,\cdot): X \rightarrow X$ is continuous or differentiable for all $t\in \mathbb{R}^+$ and $P$-a.e. $\omega \in \Omega$.
\end{defn}
\begin{defn}\label{defn3}A set-valued map $\Omega \ni \omega \mapsto D(\omega) \in 2^{X}$ is said to be a random set in $X$ if the mapping $\omega \mapsto d(x, D(\omega))$ is $(\mathcal{F}, \mathcal{B}(\mathbb{R}))$-measurable for any $x \in X,$ where $d(x, D(\omega))\triangleq \inf _{y\in  D(\omega)} \mathrm{d}(x, y)$ is the distance in $X$ between the element $x$ and the set $D(\omega)  \subset X$.
\end{defn}
\begin{defn}\label{defn4}A random set $\{D(\omega)\}_{\omega \in \Omega}$ of $X$ is called tempered with respect to $\{\theta_{t}\}_{t \in \mathbb{R}}$ if for $P$-a.e. $\omega \in \Omega$,
$$
\lim _{t \rightarrow \infty} e^{-\beta t} d\left(D\left(\theta_{-t} \omega\right)\right)=0, \quad \text { for all } \beta>0,
$$
where $d(D)=\sup _{x \in D}\|x\|_{X}$.
\end{defn}

\begin{defn}\label{defn5}
Let $\mathcal{D}=\{D(\omega)\subset X, \omega\in\Omega\}$ be a family of random set.  A random set $K(\omega) \in \mathcal{D}$ is said to be a $\mathcal{D}$-pullback absorbing set for $\Phi$ if for $P$-a.e. $\omega \in \Omega$ and for every $B \in \mathcal{D},$ there exists $T=T(B, \omega)>0$ such that
\[
\Phi\left(t, \theta_{-t} \omega,B\left(\theta_{-t} \omega\right)\right)  \subseteq K(\omega) \quad \quad \text { for all } t \geq T.
\]
If, in addition, for all  $\omega \in \Omega, K(\omega)$ is a closed nonempty subset of $X$ and $K(\omega)$ is measurable in $\Omega$ with respect to $\mathcal{F},$ then we say $K$ is a closed measurable $\mathcal{D}$-pullback absorbing set for $\Phi$.
\end{defn}

\begin{defn}\label{defn6}
 A RDS $\Phi$ is said to be $\mathcal{D}$-pullback asymptotically compact in $X$ if for $P$-a.e. $\omega \in \Omega$, $\left\{\Phi\left(t_{n}, \theta_{-t_{n}} \omega, x_{n}\right)\right\}_{n \geq 1}$ has a convergent subsequence in $X$ whenever $t_{n} \rightarrow \infty$ and $x_{n} \in D\left(\theta_{-t_{n}} \omega\right)$ for any given $D \in \mathcal{D}$.
\end{defn}

\begin{defn}\label{defn7}
A compact random set $\mathcal{A}(\omega)$ is said to be a $\mathcal{D}$-pullback random attractor associated to the RDS
$\Phi$ if it satisfies the invariance property
$$\Phi(t, \omega) \mathcal{A}(\omega)=\mathcal{A}\left(\theta_{t} \omega\right), \quad \text { for all } t \geq 0, $$
and the pullback attracting property
\[
\lim _{t \rightarrow \infty} \operatorname{dist}\left(\Phi\left(t, \theta_{-t} \omega\right)D\left(\theta_{-t} \omega\right), \mathcal{A}(\omega)\right)=0, \quad \text { for all } t \geq 0, D \in \mathcal{D}, P-a.e.\  \omega\in \Omega.
\]
where  $\operatorname{dist} (\cdot, \cdot)$ denotes the Hausdorff semidistance
\[
\operatorname{dist}(A, B)=\sup _{x \in A} \inf _{y\in  B} \mathrm{d}(x, y), \quad A, B \subset  X.
\]
\end{defn}

\begin{lem}\label{lem1}
Let $(\theta, \Phi)$ be a continuous random dynamical system. Suppose that $\Phi$ is $\mathcal{D}$-pullback asymptotically compact and has a closed pullback $\mathcal{D}$-absorbing set $K=\{K(\omega)\}_{\omega \in \Omega} \in \mathcal{D}$. Then it possesses a random attractor $\{\mathcal{A}(\omega)\}_{\omega \in \Omega},$ where
$$
\mathcal{A}(\omega)=\cap_{\tau \geq 0} \overline{\cup_{t \geq \tau} \Phi\left(t, \theta_{-t} \omega, K\left(\theta_{-t} \omega\right)\right)}.
$$
\end{lem}

For convenience, we introduce the following Gr{o}nwall inequality in \cite{17} that will be frequently used in our subsequent proofs.
\begin{lem}\label{lem2}
Let $T>0$ and $u, \alpha, f$ and $g$ be non-negative continuous functions defined on $[0, T]$ such that
\[
u(t) \leq \alpha(t)+f(t) \int_{0}^{t} g(r) u(r) d r, \quad \text { for } t \in[0, T].
\]
Then
\[
u(t) \leq \alpha(t)+f(t) \int_{0}^{t} g(r) \alpha(r) e^{\int_{r}^{t} f(\tau) g(\tau) d \tau} d r, \quad \text { for } t \in[0, T].
\]
\end{lem}

\section{Global solutions and random dynamical systems}
In this section, we will prove the existence of  global solutions to SNDRDE \eqref{4} under the given initial condition, and then show that the solutions generate  a random dynamical system. By the Fourier sine transform  defined by Eq. (10.5.39)  in  \cite{HR}, we can obtain that the semigroup $S(t)$ generated by the linear system
 \begin{equation}\label{3.1}
\left\{\begin{array}{l}
\frac{\partial u}{\partial t}=\Delta u-\mu u, \quad t>0 \\
u(t, 0)=0, \quad t \geq 0 \\
u(0, x)=\phi(x), \quad x \in \mathbb{R}_{+}
\end{array}\right.
 \end{equation}
is
 \begin{equation}\label{3.2}
\left\{\begin{array}{l}
S(0)[\phi](x)=\phi(x), \\
S(t)[\phi](x)=\frac{\exp (-\mu t)}{\sqrt{4 \pi t}} \int_{0}^{\infty} \phi(y)\left[\exp \left(-\frac{(x-y)^{2}}{4 t}\right)-\exp \left(-\frac{(x+y)^{2}}{4 t}\right)\right] \mathrm{dy}, \quad  t>0,
\end{array}\right.
 \end{equation}
for $(x, \phi) \in \mathbb{R}_{+} \times X$. Let $Z=BUC(\mathbb{R}, \mathbb{R})$ be the set of all bounded and uniformly continuous functions from $\mathbb{R}$ to $\mathbb{R}$ equipped with the usual supremum norm $\|\cdot\|_{Z}$. Then, the  Fourier transformation method indicates that the semigroup  $U(t): Z \rightarrow Z$ generated by $\Delta-\mu I$ is defined as
 \begin{equation}\label{3.3}
\left\{\begin{array}{l}
U(0)[\phi](x)=\phi(x), \\
U(t)[\phi](x)=\frac{\exp (-\mu t)}{\sqrt{4 \pi t}} \int_{-\infty}^{\infty} \phi(y) \exp \left(-\frac{(x-y)^{2}}{4 t}\right) \text { dy for all } t \in(0, \infty),
\end{array}\right.
 \end{equation}
for $(x, \phi) \in \mathbb{R} \times Z$, which is  analytic and strongly continuous on $Z$.

We introduce the following results concerning the properties of semigroup $S(t)$, which is frequently used thorough the whole paper. The details of the proof can be found in \cite{YZD} Lemma 2.1.
\begin{lem}\label{lem2.1}
Let $S(t)$ and $U(t)$ be defined in \eqref{3.2} and \eqref{3.3} respectively, then we have the following results.\\
(i) $S(t)[\phi](x)=e^{-\mu t} U(t)[\tilde{\phi}](x)$ for all $\phi \in X$, $t \in \mathbb{R}_{+}$ and $x \in \mathbb{R}_{+}$, where $\tilde{\phi}$ represents the odd extension of $\phi$.\\
(ii) $S(t)$ is an analytic and strongly continuous semigroup on $X$.\\
(iii) For all $t \in(0, \infty)$ and $(x, \phi) \in (0, \infty)\times X,$ there hold
$$
\begin{array}{l}
\|S(t)[\phi]\|\leq e^{-\mu t}\|\phi\|, \quad\left|\frac{\partial S(t)[\phi](x)}{\partial t}\right| \leq \frac{(1+\mu t) \exp (-\mu t)\|\phi\|}{t}, \\
\left|\frac{\partial S(t)[\phi](x)}{\partial x}\right| \leq \frac{\exp (-\mu t)\|\phi\|}{\sqrt{\pi t}}, \quad\left|\frac{\partial^{2} S(t)[\phi](x)}{\partial x^{2}}\right| \leq \frac{\exp (-\mu t)\|\phi\|}{t}.
\end{array}
$$
(iv) For any $t_{1}, t_{2} \in(0, \infty), x_{1}, x_{2} \in \mathbb{R}_{+}$ and $\phi \in X,$ there hold
$$
\begin{aligned}
\left|S\left(t_{1}\right)[\phi]\left(x_{1}\right)-S\left(t_{2}\right)[\phi]\left(x_{2}\right)\right| \leq & \frac{\left(1+\mu \min \left\{t_{1}, t_{2}\right\}\right) \exp \left(-\mu \min \left\{t_{1}, t_{2}\right\}\right)|| \phi||}{\min \left\{t_{1}, t_{2}\right\}}\left|t_{2}-t_{1}\right| \\
&+\frac{\exp \left(-\mu \min \left\{t_{1}, t_{2}\right\}\right)|| \phi||}{\sqrt{\pi \min \left\{t_{1}, t_{2}\right\}}}\left|x_{2}-x_{1}\right|.
\end{aligned}
$$
\end{lem}

For the purpose of later use, we prove the following property on the nonlocal diffusion operator of \eqref{4}.
\begin{lem}\label{lem3.2}
Define $K: X \rightarrow X$ by
$$
K(\phi)(\cdot)=\int_{\mathbb{R}_+} \Gamma(\alpha, \cdot, y) \phi(y) \mathrm{d} y
$$
for all  $\phi \in X$. Then,  $\|K\|\triangleq \sup\{\frac{\|K(\phi)\|}{\|\phi\|}: \|\phi\|\neq 0\}\leq 1$.
\end{lem}
\begin{proof}
For any $x\in \mathbb{R}_+$, we have
\begin{equation}\label{3.3b}
\begin{aligned}
|K(\phi)(x)|=&|\int_{\mathbb{R}_+}\frac{1}{\sqrt{4\pi \alpha}}e^{-\frac{(x-y)^2}{4\alpha}}\phi(y) \mathrm{d} y-\int_{\mathbb{R}_+}\frac{1}{\sqrt{4\pi \alpha}}e^{-\frac{(x+y)^2}{4\alpha}}\phi(y) \mathrm{d}y|\\
=&|\int_{-\infty}^x\frac{1}{\sqrt{4\pi \alpha}}e^{-\frac{u^2}{4\alpha}}\phi(x-u) \mathrm{d} u-\int_x^\infty\frac{1}{\sqrt{4\pi \alpha}}e^{-\frac{u^2}{4\alpha}}\phi(u-x) \mathrm{d}u|\\
&\leq \|\phi\|\int_{-\infty}^\infty\frac{1}{\sqrt{4\pi \alpha}}e^{-\frac{u^2}{4\alpha}} \mathrm{d} u=\|\phi\|.
\end{aligned}
\end{equation}
Therefore, we have $\|K\|\triangleq \sup\limits_{\|\phi\|\neq 0}\ \frac{\|K(\phi)\|}{\|\phi\|}\leq 1$.
\end{proof}

In the sequel, we always impose the following assumptions on the nonlinear drift term $f$.\\

$\left(\mathbf{H}\right) f(\cdot):\mathbb{R} \rightarrow \mathbb{R}$ is continuously differentiable, $f(0)=0$, and $ \left|\frac{df(s)}{ds}\right| \leq L_{f}$.\\

Here,  $L_{f}$ is a positive constant, representing  the bound of the  derivation. Hence, it is clear that for all  $s_{1}, s_{2} \in \mathbb{R}$
 \begin{equation}\label{3.3a}
\begin{array}{c}
\left|f\left(s_{1}\right)-f\left( s_{2}\right)\right| \leq L_{f}\left|s_{1}-s_{2}\right|,
\end{array}
\end{equation}

In this paper, we consider the canonical probability space $(\Omega, \mathcal{F}, P)$  with
$$
\Omega=\left\{\omega=\left(\omega_{1}, \omega_{2}, \ldots, \omega_{m}\right) \in C\left(\mathbb{R} ; \mathbb{R}^{m}\right): \omega_i(0)=0\right\}
$$
and $\mathcal{F}$ is the Borel $\sigma$-algebra induced by the compact open topology of $\Omega,$ while $P$ is the corresponding Wiener measure on $(\Omega, \mathcal{F})$. Then, we identify $\omega$ with
$$
W(t,\omega)=\left(w_{1}(t), w_{2}(t), \ldots, w_{m}(t)\right) \quad \text { for } t \in \mathbb{R}.
$$
Moreover, we define the time shift by $$\theta_{t} \omega(\cdot)=\omega(\cdot+t)-\omega(t), t \in \mathbb{R}.$$
Then, $\left(\Omega, \mathcal{F}, P,\left\{\theta_{t}\right\}_{t \in \mathbb{R}}\right)$ is a metric dynamical system.

In order to construct the conjugate transformation, we consider the stochastic stationary solution of the one dimensional Ornstein-Uhlenbeck equation
\begin{equation}\label{3.5}
\mathrm{d} z_{j}+\mu z_{j} \mathrm{d} t=\mathrm{d} w_{j}(t), j=1, \ldots, m.
\end{equation}

The solution to \eqref{3.5} is given by
\begin{equation}\label{3.6}
z_{j}(t) \triangleq z_{j}\left(\theta_{t} \omega_{j}\right)=-\mu \int_{-\infty}^{0} e^{\mu s}\left(\theta_{t} \omega_{j}\right)(s) \mathrm{d} s, \quad t \in \mathbb{R},
\end{equation}

By Definition \ref{defn4}, one can see that the random variable $\left|z_{j}\left(\omega_{j}\right)\right|$ is tempered and $z_{j}\left(\theta_{t} \omega_{j}\right)$ is $P$-a.e. $\omega$ continuous. Therefore,  Proposition 4.3.3 in \cite{AL} implies that there exists a tempered function $0<r(\omega)<\infty$ such that
\begin{equation}\label{3.7}
\sum_{j=1}^{m}\left|z_{j}\left(\omega_{j}\right)\right|^{2} \leq r(\omega),
\end{equation}
where $r(\omega)$ satisfies, for $P$-a.e. $\omega \in \Omega$,
\begin{equation}\label{3.8}
r\left(\theta_{t} \omega\right) \leq e^{\frac{\mu}{2}|t|} r(\omega), \quad t \in \mathbb{R}.
\end{equation}
Combining \eqref{3.7} with \eqref{3.8}, we obtain that  for $P$-a.e. $\omega \in \Omega$,
\begin{equation}\label{3.9}
\sum_{j=1}^{m}\left|z_{j}\left(\theta_t\omega_{j}\right)\right|^{2} \leq e^{\frac{\mu}{2}|t|} r(\omega), \quad t \in \mathbb{R}
\end{equation}

Putting $z\left(\theta_{t} \omega\right)(x)=\sum_{j=1}^{m} g_{j}(x) z_{j}\left(\theta_{t} \omega_{j}\right),$ by \eqref{3.5}, we have
$$
\mathrm{d} z+\mu z \mathrm{d} t=\sum_{j=1}^{m} g_{j}(x) \mathrm{d}w_{j}.
$$
To prove that \eqref{4} possesses a global solution that generates a RDS, we consider the transformation $v(t)=u(t)-z\left(\theta_{t} \omega\right)$, where $u$ is a solution of  \eqref{4}, and show that $v$ is a global solution of the transformed equation and  generates a random dynamical system. Then, we show that \eqref{4} also has a global solution and generates a conjugated RDS thanks to the inverse transformation. This method has also been adopted by \cite{HZ20}, \cite{LG} and \cite{25}  to deal with random attractors as well as  \cite{DLS03,DLS04,LS07} and \cite{LS08} to tackle invariant manifolds of stochastic partial differential equations with or without delay. Apparently, $v$ satisfies
\begin{equation}\label{3.10}
\displaystyle \frac{\partial v(t,x)}{\partial t}=\displaystyle   \Delta v(t,x)-\mu v(t,x)+F(\left(v_t+z\left(\theta_{t+\cdot} \omega\right)\right))(x)+\Delta z\left(\theta_{t} \omega\right)(x), t >0,x\in (0,\infty)
\end{equation}
with boundary condition
\begin{equation}\label{3.11}
v(t, 0)=0, \quad \text { for } \quad t \in(0, \infty),
\end{equation}
and initial condition
\begin{equation}\label{3.12}
v(\xi, x, \omega)=\psi(\xi, x, \omega) \triangleq \phi(x, \xi)-z\left(\theta_{\xi} \omega\right)(x) \quad \text { for } \quad(x, \xi) \in \mathbb{R}_{+} \times[-\tau, 0].
\end{equation}
Here, $F: \mathcal{C} \rightarrow X $ is defined by
$$
F(\varphi_t+z\left(\theta_{t+\cdot} \omega\right))(x)=\varepsilon\int_{\mathbb{R}_{+}}\Gamma(\alpha,x,y) f(\varphi(t-\tau, y)+z\left(\theta_{t-\tau}\omega\right))(y)\mathrm{dy}=\varepsilon K[f(\varphi(t-\tau, \cdot)+z\left(\theta_{t-\tau}\omega, \cdot\right))] (x),
$$
for any $\varphi\in \mathcal{C}$.

We now show that the pathwise deterministic problem \eqref{3.10}-\eqref{3.12} has a global mild solution under assumption  $\left(\mathbf{H}\right)$. We aim at solving  the following integral equation
 \begin{equation}\label{3.13}
v(t,\omega,\psi)=\left\{\begin{array}{l}
S(t) \psi(0)+\int_{0}^{t} S(t-r) F\left(v_{r}+z\left(\theta_{r+\cdot} \omega\right)\right) d r+\int_{0}^{t} S(t-r) \Delta z\left(\theta_{r} \omega\right) d r,\\
\psi(t), t \in[-\tau, 0],
\end{array}\right.
 \end{equation}
for the initial data $\psi \in \mathcal{C}$. We have the following results.
\begin{thm}\label{thm1}
Assume that $f$ satisfies $\left(\mathbf{H}\right)$. Then, for any $\psi \in \mathcal{C}$ and for $P$-a.e. $\omega \in \Omega$, there exists a global mild solution to \eqref{3.10}-\eqref{3.12}. Moreover, if $f: \mathcal{C}\rightarrow X$ is globally bounded, i.e. there exists $M>0$ such that $\|f(\varphi)\|\leq M$ for all $\varphi \in \mathcal{C}$, then the solution is pullback bounded, i.e., there exists $C(\omega)>0$ such that $\|v(t, \theta_{-t}\omega, \psi)\|\leq C(\omega)$  for $P$-a.e. $\omega \in \Omega$.
\end{thm}
\begin{proof}
We first prove that \eqref{3.10}-\eqref{3.12} has a local mild solution and then show it can be extended to a global one by an argument of steps.  For any $\psi \in \mathcal{C}$ and $P$-a.e. $\omega \in \Omega$, we show in the following that there exist  $T(\omega)>0$ and $v \in C([-\tau, T(\omega)]; X)$ satisfying \eqref{3.13} on $[-\tau, T(\omega)]$ thanks to the Banach fixed point theorem.{For a fixed $\omega$, we consider the complete metric subset $X^T_{\psi}$ of $C([-\tau, T], X)$ defined by
$$
X_{\psi}^T=\{v \in C([-\tau, T] ; X): u(s)=\psi(s), s \in[-\tau, 0]\}.
$$}
For such a $T>0$ to be determined later, and $t\in [-\tau,T]$, we define the following operator $\Lambda: X_{\psi} \rightarrow X_{\psi}$ (where we omit $T$ since no confusion is possible)
 \begin{equation}\label{3.14}
\Lambda(\zeta)(t)=\left\{\begin{array}{l}
S(t) \psi(0)+\int_{0}^{t}S(t-r) F\left(\zeta_{r}+z\left(\theta_{r+\cdot} \omega\right)\right) d r+\int_{0}^{t} S(t-r) \Delta z\left(\theta_{r} \omega\right) dr, t \in (0, T]\\
\psi(t), t \in[-\tau, 0].
\end{array}\right.
 \end{equation}
We show in the sequel that $\Lambda$ is well defined, maps  $X_{\psi}$ into itself and is a contraction on $C([-\tau, T]; X)$, leading to the existence of  a unique fixed point in $X_{\psi}$ with $T$ being determined according to the Banach fixed point theorem. It follows from Lemma \ref{lem2.1} (ii) and (iii), $F:\mathcal{C}\rightarrow X$ and $g_j $ is twice continuously differentialble that we have $\Lambda(\zeta)(t)\in X$ for any fixed $t\in [-\tau, T]$. Now we prove the continuity. If $t_1, t_2\in[-\tau,0]$, the result is obvious. Let us then pick $t_1, t_2\in (0, T]$, and assume without loss of  generality, that $t_1<t_2$. Therefore, we have
\begin{equation}\label{3.15}
\begin{aligned}
\|\Lambda(\zeta)(t_1)-\Lambda(\zeta)(t_2)\|&=\| [S(t_1)-S(t_2)]\psi(0)\|+\|\int_{0}^{t_1}S(t_1-r) F\left(\zeta_{r}+z\left(\theta_{r+\cdot} \omega\right)\right) dr\\
&-\int_{0}^{t_2}S(t_2-r) F\left(\zeta_{r}+z\left(\theta_{r+\cdot} \omega\right)\right) dr\|+\|\int_{0}^{t_1} S(t_1-r) \Delta z\left(\theta_{r} \omega\right)dr\\
&-\int_{0}^{t_2} S(t_2-r) \Delta z\left(\theta_{r} \omega\right)\|dr
\triangleq I_1+I_2+I_3.
\end{aligned}
\end{equation}
We estimate each term on the right hand side of \eqref{3.15} thanks to Lemmas \ref{lem2.1} and 3.2.
\begin{equation}\label{3.16}
\begin{aligned}
I_1=\|S(t_1)-S(t_2)]\psi(0)\|&\leq \frac{\left(1+t_{1}\right) \exp \left(-\mu t_{1}\right)||\psi(0)||}{t_{1}}\left|t_{2}-t_{1}\right|.
\end{aligned}
\end{equation}

\begin{equation}\label{3.17}
\begin{aligned}
I_2&=\|\int_{0}^{t_1}[S(t_1-r)-S(t_2-r)] F\left(\zeta_{r}+z\left(\theta_{r+\cdot} \omega\right)\right) dr-\int_{t_1}^{t_2}S(t_2-r) F\left(\zeta_{r}+z\left(\theta_{r+\cdot} \omega\right)\right) dr]\|\\
& \leq \int_{0}^{t_1-\sqrt{\delta}}[S(t_1-r)-S(t_2-r)] F\left(\zeta_{r}+z\left(\theta_{r+\cdot} \omega\right)\right) dr\\
&+\int_{t_1-\sqrt{\delta}}^{t_1}[S(t_1-r)-S(t_2-r)] F\left(\zeta_{r}+z\left(\theta_{r+\cdot} \omega\right)\right) dr+\|\int_{t_1}^{t_2}S(t_2-r) F\left(\zeta_{r}+z\left(\theta_{r+\cdot} \omega\right)\right) dr]\|\\
& \leq \varepsilon \left|t_{2}-t_{1}\right|\int_{0}^{t_1-\sqrt{\delta}}\frac{\left(1+\mu (t_{1}-r)\right) \exp \left(-\mu (t_{1}-r)\right)M}{t_{1}-r}dr+2\varepsilon M\sqrt{\delta}+\varepsilon M|t_2-t_1|\\
& \leq \varepsilon M\left|t_{2}-t_{1}\right|(\frac{1}{\sqrt{\delta}}+\mu)+2\varepsilon M\sqrt{\delta}+\varepsilon M|t_2-t_1|.
\end{aligned}
\end{equation}

\begin{equation}\label{3.18}
\begin{aligned}
I_3&=\|\int_{0}^{t_1}[S(t_1-r)-S(t_2-r)] \Delta z\left(\theta_{r} \omega\right)dr-\int_{t_1}^{t_2}S(t_2-r) \Delta z\left(\theta_{r} \omega\right) dr]\|\\
& \leq \int_{0}^{t_1-\sqrt{\delta}}[S(t_1-r)-S(t_2-r)] ||\Delta z\left(\theta_{r} \omega\right)||dr+\int_{t_1-\sqrt{\delta}}^{t_1}[S(t_1-r)-S(t_2-r)] ||\Delta z\left(\theta_{r} \omega\right)|| dr\\
&+\|\int_{t_1}^{t_2}S(t_2-r)||\Delta z\left(\theta_{r} \omega\right)|| dr\|\\
& \leq \left|t_{2}-t_{1}\right|\int_{0}^{t_1-\sqrt{\delta}}\frac{\left(1+\mu (t_{1}-r)\right) \exp \left(-\mu (t_{1}-r)\right)||\Delta z\left(\theta_{r} \omega\right)||}{t_{1}-r}dr+\delta||\Delta z\left(\theta_{r} \omega\right)||+||\Delta z\left(\theta_{r} \omega\right)|| |t_2-t_1|\\
& \leq ||\Delta z\left(\theta_{r} \omega\right)||[\left|t_{2}-t_{1}\right|(\frac{1}{\sqrt{\delta}}+1+\mu)+\delta].
\end{aligned}
\end{equation}
In equations \eqref{3.17} and \eqref{3.18}, $\delta$ satisfies  $\delta\in(0,1)$ and $t_1<t_2<t_1+\delta$ with $\delta \rightarrow 0$, then $I_1\leq \varepsilon M (\sqrt{\delta}+\mu\delta)+2\varepsilon M\sqrt{\delta}+\varepsilon M\delta$, $I_3\leq ||\Delta z\left(\theta_{r} \omega\right)||[\sqrt{\delta}+(2+\mu)\delta]$. Hence, when $\delta \rightarrow 0$, it holds that $t_2\rightarrow t_1$, $I_1+I_2+I_3\rightarrow 0$, implying the continuity of $\Lambda(\zeta)(t)\in H$ with respect to  $t\in [-\tau, T]$. Thus, we have obtained that  $\Lambda$ is well defined in $X_{\psi}$.

In the sequel, we show the contraction property of $\Lambda$ on $X_{\psi}$. Let $\zeta^{1}, \zeta^{2} \in X_{\psi}$, then for $t\in [-\tau, 0]$, it holds $\zeta^{1}(t)=\zeta^2(t)$.  Owing to $(\mathbf{H})$, for $t \in[0, T)$ we have
\begin{equation}\label{3.19}
\begin{aligned}
\left\|\left(\Lambda\left(\zeta^1\right)(t)-\Lambda\left(\zeta^2\right)\right)(t)\right\|&= \|\int_{0}^{t}S(t-r) [F\left(\zeta^1_{r}+z\left(\theta_{r+\cdot} \omega\right)\right)-F\left(\zeta^2_{r}+z\left(\theta_{r+\cdot} \omega\right)\right)] d r\|. \\
& \leq \varepsilon L_{f}\int_{0}^{t}e^{-\mu (t-r)}\left\|\zeta_r^1-\zeta_r^2\right\|_{\mathcal{C}}dr\\
&\leq \frac{\varepsilon L_{f}}{\mu}(1-e^{-\mu t})\left\|\zeta_r^1-\zeta_r^2\right\|_{\mathcal{C}}.
\end{aligned}
\end{equation}
Hence, if $\frac{\varepsilon L_{f}}{\mu}\leq1$, then for any $t\geq 0$,   $\Lambda$ is a contraction on $X_{\psi}$. However, in the scenario $\frac{\varepsilon L_{f}}{\mu}>1$, we can choose $T=\frac{1}{2\mu}\ln(\frac{\varepsilon L_{f}}{\varepsilon L_{f}-\mu})$, and therefore $\Lambda$ is a contraction on $X_{\psi}$, which indicates the existence of a unique local mild solution to \eqref{3.13}.

In the following, we  will derive the existence of a global mild solution by an argument of steps. Denote $T_{1}(\omega)=\frac{1}{2\mu}\ln(\frac{\varepsilon L_{f}}{\varepsilon L_{f}-\mu})$ and let us build the solution in the next time interval, say $\left[T_{1}(\omega), T_{2}(\omega)\right]$. It suffices  to find $T_{2}(\omega)$ such that \eqref{3.13} also admits a local mild solution in the last interval. We only need to solve
 \begin{equation}\label{3.20}
v(t,\omega,\psi)=\left\{\begin{array}{l}
S(t-T_{1}(\omega)) \psi(0)+\int_{0}^{t-T_{1}(\omega)} S(t-T_{1}(\omega)-r) F\left(v_{r}+z\left(\theta_{r+\cdot} \omega\right)\right) d r\\
\quad \ \ \ \ +\int_{0}^{t} S(t-T_{1}(\omega)-r) \Delta z\left(\theta_{r} \omega\right) d r,\\
v_1(t), t-T_{1}(\omega) \in[-\tau, 0],
\end{array}\right.
 \end{equation}
where $v_{1}$ denotes the solution obtained on $\left[-\tau, T_{1}(\omega)\right]$. Taking $s:=t-T_{1}$, the above system is equivalent to solve the problem for $y(s)=v\left(s+T_{1}(\omega)\right)$
 \begin{equation}\label{3.21}
y(s,\omega,\psi)=\left\{\begin{array}{l}
S(s) \hat{\psi}(0)+\int_{0}^{s} S(s-r) F\left(v_{r}+z\left(\theta_{r+\cdot} \omega\right)\right) d r+\int_{0}^{s} S(s-r) \Delta z\left(\theta_{r} \omega\right) d r,\\
\hat{\psi}(s)\triangleq u_1(s+T_{1}(\omega)), s\in[-\tau, 0],
\end{array}\right.
 \end{equation}
which is the same as  the previous step, but with initial condition $\hat{\psi}$. Taking the same steps as before, we can obtain a new piece given by a local solution defined now in the interval $\left[T_1-\tau, T_2 \right]$. {Thus, by repeating the same procedure, one can obtain a sequence of time $T_n$. We prove in the sequel that $T_n\rightarrow \infty$. We only need to show that  for any given $t>0$, there exist $i\in \mathbb{N}$ such that $T_i>t$. If $T(\omega):=T_{1}(\omega) \geq t$ there is nothing to show. If this is not case, let  $s^{*}$ be  the unique  solution of the equation
$$
 \frac{\varepsilon L_{f}}{\mu}(1-e^{-\mu t})=1 / 2,
$$
which is trivially a positive lower bound of $T(\omega)$.  If $T_{2}(\omega) \geq t$ we are done. Otherwise, $t>T_{2}(\omega)=T(\omega)+T\left(\theta_{T(\omega)} \omega\right)$, i.e. $T\left(\theta_{T(\omega)} \omega\right)<$ $t-T(\omega)$, and therefore the previous inequality implies that $t^{*} \leq T\left(\theta_{T(\omega)} \omega\right)$, and in particular that $T_{2}(\omega) \geq 2 t^{*}$. Repeating this method it turns out that there exists $i \in \mathbb{N}$ such that $T_{i}(\omega) \geq i t^{*}>t$.}

In what follows, we prove the pullback boundedness of the solution provided $f$ is bounded.  Since $g_j$ is twice continuously differentiable, by \eqref{3.8} and \eqref{3.9}, there must exist a constant $c>0$ such that $\|\Delta z(\theta_{-t}\omega)\|\leq ce^{\frac{\mu t}{2}}r(\omega)$. It follows from \eqref{3.13} and the boundedness of $f$ that, for $P$-a.e. $\omega \in \Omega$,
\begin{equation}\label{3.35}
\begin{aligned}
\left\|v(t,\theta_{-t}\omega,\psi)\right\|&= |S(t)|\| \psi(0)\|+M\int_{0}^{t}e^{-\mu (t-r)}d r+c\int_{0}^{t}e^{-\mu (t-r)}e^{\frac{\mu(t-r)}{2}}r(\omega)d r\\
& \leq e^{-\mu t}\| \psi(0)\|+M\frac{1}{\mu}(1-e^{-\mu t})+cr(\omega)\frac{2}{\mu}(1-e^{-\mu t/2})\\
&\leq \| \psi(0)\|+(M+cr(\omega))\frac{2}{\mu}.
\end{aligned}
\end{equation}
Therefore, the pullback boundedness of $v$ is clear by taking $C(\omega)=\| \psi(0)\|+(M+cr(\omega))\frac{2}{\mu}$.
\end{proof}

\begin{rem}\label{rem2}
By Corollary 2.2.5 in \cite{W} and the analyticity of the semigroup $S(t)$ given in Lemma \ref{lem2.1} (ii), we know that a
mild solution of problem \eqref{3.10}-\eqref{3.12} is also a classical solution of problem \eqref{3.10}-\eqref{3.12} for all $t>\tau$. Hence, $u(t, \omega, \phi)=v(t, \omega, \psi)+z\left(\theta_{t} \omega\right)$ is a global solution  to \eqref{4}.
\end{rem}

In the sequel, we show that the  solution of \eqref{3.13} generates a RDS. To this end we will prove that the cocycle property in Definition \ref{defn2} holds.

\begin{thm}\label{thm3.2}
The global mild solution $v$ of \eqref{3.10}-\eqref{3.12} generates a random dynamical system $\Phi: \mathbb{R}^{+} \times \Omega \times \mathcal{C} \rightarrow \mathcal{C}$
defined by $\Phi(t, \omega, \psi)(\cdot)=v_{t}(\cdot),$ i.e.,\\
 \begin{equation}\label{3.22}
\Phi(t, \omega, \psi)(\cdot)=\left\{\begin{array}{l}
S(t+\cdot) \psi(0)+\int_{0}^{t+\cdot} S(t+\cdot-r) F\left(v_{r}+z\left(\theta_{r-\tau+\cdot} \omega\right)\right) d r+\int_{0}^{t+\cdot} S(t-r+\cdot) \Delta z\left(\theta_{r} \omega\right) d r,\\
\psi(t+\cdot), t+\cdot \in[-\tau, 0].
\end{array}\right.
 \end{equation}
\end{thm}
\begin{proof} We prove the result in three  cases. In the situation  $t, \rho \geq \tau$ so that $t+s, \rho+s \geq 0,$ for all $s \in[-\tau, 0]$, we have
 \begin{equation}\label{3.23}
\begin{aligned}
\Phi(t+\rho, \omega, \psi)(\zeta)=&
S(t+\zeta+\rho)\psi(0)+\int_{0}^{t+\zeta+\rho} S(t+\zeta+\rho-r) [F\left(v_{r}+z\left(\theta_{r+\cdot} \omega\right)\right)+\Delta z\left(\theta_{r} \omega\right)] d r\\
&=S(t+\zeta) S(\rho)\psi(0)+S(t+\zeta) \int_{0}^{\rho} S(\rho-r) [F\left(v_{r}+z\left(\theta_{r+\cdot} \omega\right)\right) +\Delta z\left(\theta_{r} \omega\right)] d r\\
&+\int_{\rho}^{t+\zeta+\rho}  S(t+\zeta+\rho-r) [F\left(v_{r}+z\left(\theta_{r+\cdot} \omega\right)\right) +\Delta z\left(\theta_{r} \omega\right)] d r\\
&=S(t+\zeta) [S(\rho)\psi(0)+\int_{0}^{\rho} S(\rho-r) F\left(v_{r}+z\left(\theta_{r+\cdot} \omega\right)\right) d r+\int_{0}^{\rho} S(\rho-r) \Delta z\left(\theta_{r} \omega\right) d r]\\
&+\int_{0}^{t+\zeta} S(t+\zeta-r) F\left(v_{\rho+r}+z\left(\theta_{r+\cdot+\rho} \omega\right)\right) d r+\int_{0}^{t+\zeta} S(t+\zeta-r) \Delta z\left(\theta_{r+\rho} \omega\right) d r\\
&=S(t+\zeta)\Phi(\rho, \omega, \psi)(0)+\int_{0}^{t+\zeta} S(t+\zeta-r) F\left(v_{\zeta+r}+z\left(\theta_{r+\cdot}\theta_\rho \omega\right)\right) d r\\
&+\int_{0}^{t+\zeta} S(t+\zeta-r) \Delta z\left(\theta_r\theta_\rho \omega\right) d r\\
&=\Phi(t, \theta_\rho\omega, \cdot)\Phi(\rho, \omega, \psi)(\zeta),
\end{aligned}
 \end{equation}
which indicates  the cocycle property in this situation.

In the scenario $t+\rho+\zeta \leq 0,$ for $\zeta \in[-\tau, 0]$. Then, it is straightforward to see that
$$
\Phi(t+\rho, \omega, \psi)(\zeta)=\psi(t+\rho+\zeta)=\Phi(\rho, \omega, \psi)(t+\zeta)=\Phi\left(t, \theta_\rho\omega, \cdot\right) \circ \Phi(\rho, \omega, \psi)(\zeta)
$$

When  $t+\rho+\zeta > 0$ but  $\rho+\zeta \leq 0 $ for $\zeta \in[-\tau, 0]$, we have
$$
\psi(\rho+\zeta)(0)=v_{\rho+\zeta}(0)=\Phi(\zeta+\rho, \omega, \psi)(0)=\Phi\left(\rho, \omega, \psi\right)(\zeta)
$$
Moreover, by \eqref{3.22}, one can easily check that
$$\Phi(t+\rho, \omega, \psi)(\zeta)=\Phi\left(t, \theta_\rho\omega, \psi(\rho+\zeta)\right)(0).$$
Therefore, we have
$$
\Phi(t+\rho, \omega, \psi)(\zeta)=\Phi(t, \theta_\rho\omega, \psi)(\rho+\zeta)=\Phi(t, \theta_\rho\omega, \psi(\rho+\zeta))(0)=\Phi\left(t, \theta_\rho\omega, \cdot\right) \circ \Phi(\rho, \omega, \psi)(\zeta).
$$
\end{proof}

By Remark \ref{rem2}, $u(t, \omega, \phi)=v(t, \omega, \psi)+z\left(\theta_{t} \omega\right)$  is the global  solution to \eqref{4} with initial condition $\phi$.  We now define a mapping $\Psi: \mathbb{R}^{+} \times \Omega \times \mathcal{C} \rightarrow \mathcal{C}$ by
$\Psi(t, \omega, \phi)=u_{t}(\cdot, \omega, \phi)=v_{t}(\cdot, \omega,\psi)+z\left(\theta_{t+.} \omega\right)$,
where $u_{t}(\zeta, \omega, \phi)=u(t+\zeta, \omega, \phi)$ for $\zeta \in[-\tau, 0]$. By Theorem \ref{thm3.2} and the cocycle property of $z$, $\Psi$ is an $\mathrm{RDS}$ on $\mathcal{C}$ generated by  \eqref{4}.

\section{Existence of random  attractors}
In this  section, we are concerned with the existence of tempered pullback attractors for the SNDRDE \eqref{4} by first establishing a uniform estimation for the solution and then proving that $\Psi$ is $\mathcal{D}$-pullback asymptotically compact. Nevertheless, due to the non-compactness of the spatial domain, it is quite difficult to prove the asymptotically compact of $\Psi$ with respect to the usual supreme norm. Hence, similar to \cite{YCWT}, we introduce another more suitable topology called the compact open topology induced by the norms $ \|\varphi\|_{co}^X=\sum_{n\geq 1}2^{-n}\sup\{|\varphi(x)|:x\in [0, n], \ n\in \mathbb{N}\}$ for all $\varphi\in X$ and $\|\phi\|_{co}^\mathcal{C}= \sup\{\|\phi(\theta)\|_{co}^X:\theta\in [-\tau,0]\}$ for all $ \phi\in \mathcal{C}$ respectively to describe the pullback asymptotic compactness of the RDS $\Psi$ generated by \eqref{4}. Moreover, we use $X_{co}$ and $C_{co}$ to denote  the spaces $(X, \|\cdot\|_{co}^X)$ and $(C, \|\cdot\|_{co}^\mathcal{C})$ respectively.

In order to adopt compact open topology to describe the global dynamics of ($\ref{4}$), we first introduce without proof the following lemma, which gives sufficient and necessary condition for a sequence to be convergent and pre-compact with respect to the compact open topology. For details of the proof, the readers are referred to Lemma 2.1 in \cite{W}.
\begin{lem}\label{lem5.1}  Given $r>0$. Let $B_r=\{\phi\in * :\|\phi\|_*\leq r\}$  and $d_r(\phi,\psi)=\|\phi-\psi\|_{co}^{*}$, where $*$ stands for $X$ or $C$. Then the following statements are true:
\item (i) For any $\phi_n,\phi \in B_r$ with $ n\in \mathbb{N}$, $\lim_{n\rightarrow\infty}d_r(\phi_n,\phi)=0$ if and only if $$\lim_{n\rightarrow \infty}\sup\{|\phi_n (\zeta,x)-\phi(\zeta,x)|:\zeta\in[-\tau,0], x\in I\}=0$$ for any bounded domain $I=[0, i] \subset \mathbb{R_{+}}$ for all $i\in \mathbb{R}_+$.
\item (ii) Let $A\subseteq B_r$. Then $A$ is pre-compact if and only if $A_{I}=\{\varphi|_{I}:\varphi \in A\}$ is a family of equicontinuous functions for any domain $I= [0, i] \subset \mathbb{R_{+}}$.
\end{lem}

Throughout the rest of this paper, we always use $\mathcal{D}$ to denote the collection of all families of tempered nonempty subsets of  $C_{co}$. The letters $c$ and $c_i, (i=1,2,\cdots)$ are general positive
constants whose values are not significant. Moreover, as for the asymptotic behavior, we always assume that $t>\tau$ in the remaining part of this paper for convenience. The following lemma shows that the RDS $\Psi$ has a random absorbing set respect to the compact open topology.

\begin{lem}\label{lem4.1}
Assume that $\left(\mathbf{H}\right)$ is satisfied and $\varepsilon L_{f}e^{\mu \tau}-\mu<0$, then there exists $\{K(\omega)\}_{\omega \in \Omega} \in \mathcal{D}$ satisfying that, for any $B=\{B(\omega)\}_{\omega \in \Omega} \in \mathcal{D}$ and $P$-a.e. $\omega \in \Omega$, there is $T_{B}(\omega)>0$ such that
$$
\Psi \left(t, \theta_{-t} \omega, B\left(\theta_{-t} \omega\right)\right) \subseteq K(\omega) \quad \text { for all } t \geqslant T_{B}(\omega),
$$
that is, $\{K(\omega)\}_{\omega \in \Omega}$ is a random absorbing set for $\Psi$ in $\mathcal{D}$.
\end{lem}
\begin{proof}
We first derive uniform estimate on $v$ by \eqref{3.13} and then obtain the existence of absorbing set of $u$ by $u(t, \omega, \phi)=v(t, \omega, \psi)+z\left(\theta_{t} \omega\right)$. It follows from \eqref{3.13} and Lemma \ref{lem2.1} (i) that for any $t>\tau$, we have
 \begin{equation}\label{4.1}
\begin{aligned}
v(t,\omega,\psi)&=
S(t) \psi(0)+\int_{0}^{t} S(t-s) F\left(v_{s}+z\left(\theta_{s+\cdot} \omega\right)\right) ds+\int_{0}^{t} S(t-s) \Delta z\left(\theta_{s} \omega\right) ds\\
&=e^{-\mu t}U(t) \tilde{\psi}(0)+\int_{0}^{t} e^{-\mu (t-s)}U(t-s) \tilde{F}\left(v_{s}+z\left(\theta_{s+\cdot} \omega\right)\right) d s+\int_{0}^{t}e^{-\mu (t-s)}U(t-s) \Delta \tilde{z}\left(\theta_{s} \omega\right) dr,
\end{aligned}
 \end{equation}
where $ \tilde{\psi}$, $\tilde{F}$ and $\tilde{z}$ represent the odd extension of $\psi, F, z$ with respect to the spatial variable respectively. Therefore, by Lemma \ref{lem2.1}(iii) and lemma \ref{lem3.2}, for any $\zeta\in [-\tau, 0], \ n\in \mathbb{N}, x\in [0, n]$ and all $\psi \in \mathcal{C}$ we have
\begin{equation}\label{4.2}
\begin{aligned}
|v(t+\zeta,\omega,\psi)(x)|&\leq
e^{-\mu(t-\tau)}|\psi(0)(x)|+\varepsilon \int_{0}^{t+\zeta}e^{\mu (s-t-\zeta)}|K[v(s-\tau, \omega, \psi)+z(\theta_{s-\tau} \omega)](x)|ds\\
& +\int_{0}^{t+\zeta} e^{\mu (s-t-\zeta)}| \Delta z\left(\theta_{s} \omega\right)(x)|ds\\
&\leq e^{-\mu(t-\tau)}|\psi(0)(x)|+  \varepsilon L_{f}e^{\mu \tau} \int_{0}^{t} e^{\mu (s-t)}(|v(s-\tau, \omega,\psi)(x)|+|z\left(\theta_{s-\tau} \omega\right)(x)|) \mathrm{d} s\\
&+ e^{\mu \tau} \int_{0}^{t} e^{\mu (s-t)}|\Delta z\left(\theta_{s} \omega\right)(x)|\mathrm{d}s,
\end{aligned}
 \end{equation}
 which implies that
 \begin{equation}\label{4.2a}
\begin{aligned}
\sum_{n\geq 1}2^{-n}\sup_{x\in [0, n]}|v(t+\zeta,\omega,\psi)(x)|&\leq
e^{-\mu(t-\tau)}\sum_{n\geq 1}2^{-n}\sup_{x\in [0, n]}|\psi(0)(x)|+\varepsilon \int_{0}^{t+\zeta}e^{\mu (s-t-\zeta)}\\&\sum_{n\geq 1}2^{-n}\sup_{x\in [0, n]}|K[v(s-\tau, \omega, \psi)
+z(\theta_{s-\tau} \omega)](x)|ds+\\& \int_{0}^{t+\zeta} e^{\mu (s-t-\zeta)}
\sum_{n\geq 1}2^{-n}\sup_{x\in [0, n]}|\Delta z\left(\theta_{s} \omega\right)(x)|ds\\
&\leq e^{-\mu(t-\tau)}\sum_{n\geq 1}2^{-n}\sup_{x\in [0, n]}|\psi(0)(x)|+\varepsilon L_{f}e^{\mu \tau} \int_{0}^{t} e^{\mu (s-t)}\\&(\sum_{n\geq 1}2^{-n}\sup_{x\in [0, n]}|v(s-\tau, \omega,\psi)(x)|+\sum_{n\geq 1}2^{-n}\sup_{x\in [0, n]}|z\left(\theta_{s-\tau} \omega\right)(x)|) \mathrm{d} s\\&+ \int_{0}^{t} e^{\mu (s-t)}\sum_{n\geq 1}2^{-n}\sup_{x\in [0, n]}|\Delta z\left(\theta_{s} \omega\right)(x)|\mathrm{d}s.
\end{aligned}
 \end{equation}
 Therefore, by the definition of compact open topology, we have
\begin{equation}\label{4.2b}
\begin{aligned}
\|v(t+\zeta,\omega,\psi)\|_{co}^X&\leq
e^{-\mu(t-\tau)}\|\psi(0)\|_{co}^X+\varepsilon \int_{0}^{t+\zeta}e^{\mu (s-t-\zeta)}\|K[v(s-\tau, \omega, \psi)+z(\theta_{s-\tau} \omega)]\|_{co}^Xds\\
& +\int_{0}^{t+\zeta} e^{\mu (s-t-\zeta)}\| \Delta z\left(\theta_{s} \omega\right)\|_{co}^Xds\\
&\leq e^{-\mu(t-\tau)}\|\psi(0)\|_{co}^X+  \varepsilon L_{f}e^{\mu \tau} \int_{0}^{t} e^{\mu (s-t)}(\|v(s-\tau, \omega,\psi)\|_{co}^X+\|z\left(\theta_{s-\tau} \omega\right)\|_{co}^X) \mathrm{d} s\\
&+ e^{\mu \tau} \int_{0}^{t} e^{\mu (s-t)}\|\Delta z\left(\theta_{s} \omega\right)\|_{co}^X\mathrm{d}s,
\end{aligned}
 \end{equation}
for  $P$-a.e. $\omega \in \Omega$. Keep in mind that $\left\|v_{t}\right\|_{co}^{\mathcal{C}}=\sup \left\{\|v(t+\zeta)\|_{co}^X: \zeta \in[-\tau, 0]\right\}$. Hence, we can obtain
\begin{equation}\label{4.3}
\begin{aligned}
\left\|v_{t}(\cdot, \omega,\psi)\right\|_{co}^{\mathcal{C}}& \leq  e^{\mu \tau}[e^{-\mu t}\|\psi\|_{co}^{\mathcal{C}}+  \varepsilon L_{f}\int_{0}^{t} e^{\mu (s-t)}\left\|v_{s}(\cdot, \omega,\psi)\right\|_{co}^{\mathcal{C}} \mathrm{d}s\\
&+\int_{0}^{t} e^{\mu (s-t)} ( \|\Delta z\left(\theta_{s} \omega\right)\|_{co}^{X}+\varepsilon L_f \|z\left(\theta_{s-\tau} \omega\right)\|_{co}^{X})\mathrm{d} s.
\end{aligned}
\end{equation}
By replacing $\omega$ by $\theta_{-t} \omega$, we derive  from \eqref{4.3} that, for all $t \geq \tau,$
\begin{equation}\label{4.4}
\begin{aligned}
\left\|v_{t}\left(\cdot, \theta_{-t} \omega, \psi\left(\theta_{-t} \omega\right)\right)\right\|_{co}^{\mathcal{C}}\leq & e^{\mu \tau} [e^{-\mu t}\left\|\psi\left(\theta_{-t} \omega\right)\right\|_{co}^{\mathcal{C}}+  \int_{0}^{t} e^{\mu (s-t)} ( \|\Delta z\left(\theta_{s-t} \omega\right)\|_{co}^{X}+\varepsilon L_f \|z\left(\theta_{s-t-\tau} \omega\right)\|_{co}^{X}) \mathrm{d} s  \\
&+\varepsilon L_{f} \int_{0}^{t} e^{\mu(s-t)}\left\|v_{s}\left(\cdot, \theta_{-t} \omega, \psi\left(\theta_{-t} \omega\right)\right)\right\|_{co}^{\mathcal{C}} \mathrm{d}s].
\end{aligned}
\end{equation}
Since $g_j $ are twice continuously differentiable and $z\left( \omega\right)(x)=\sum_{j=1}^{m} g_{j}(x) z_{j}\left( \omega_{j}\right)$, there exists constant $c$ such that $p_1(\omega)\triangleq\|\Delta z\left(\omega\right)\|_{co}^X+\varepsilon L_f \|z\left(\theta_{-\tau}\omega\right)\|_{co}^X\leq c \sum_{j=1}^{m}\left|z_{j}\left(\omega_{j}\right)\right|^{2}$. Therefore, it follows from \eqref{3.8} and \eqref{3.9} that
\begin{equation}\label{4.5}
\begin{aligned}
\int_{0}^{t} e^{\mu (s-t)} p_{1}\left(\theta_{s-t} \omega\right) \mathrm{d} s \leq c \int_{0}^{t} e^{\frac{\mu}{2}(s-t)} r(\omega) \mathrm{d} s \leq c r(\omega).
\end{aligned}
\end{equation}
Incorporating \eqref{4.5} into \eqref{4.4}  gives rise to
\begin{equation}\label{4.6}
\begin{aligned}
\left\|v_{t}\left(\cdot, \theta_{-t} \omega, \psi\left(\theta_{-t} \omega\right)\right)\right\|_{co}^{\mathcal{C}} \leq & e^{\mu \tau} [e^{-\mu t}\left\|\psi\left(\theta_{-t} \omega\right)\right\|_{co}^{\mathcal{C}}+\varepsilon L_{f} \int_{0}^{t} e^{\mu(s-t)}\left\|v_{s}\left(\cdot, \theta_{-s} \omega, \psi\left(\theta_{-s} \omega\right)\right)\right\|_{co}^{\mathcal{C}}\mathrm{d}s+c r(\omega)].
\end{aligned}
\end{equation}
Multiply the both sides of \eqref{4.6} by $e^{\mu t}$ and adopt the Gr\"{o}nwall inequality, we have
\begin{equation}\label{4.7}
\begin{aligned}
e^{\mu t}\left\|v_{t}\left(\cdot, \theta_{-t} \omega, \psi\left(\theta_{-t} \omega\right)\right)\right\|_{co}^{\mathcal{C}}\leq & e^{\mu \tau}(\left\|\psi\left(\theta_{-t} \omega\right)\right\|_{co}^{\mathcal{C}}+ce^{\mu t}r(\omega))+ \varepsilon e^{\mu\tau}L_{f}  \int_{0}^{t}(\left\|\psi\left(\theta_{-t} \omega\right)\right\|_{co}^{\mathcal{C}}\\+
& ce^{\mu s}r(\omega))e^{\varepsilon L_{f}e^{\mu \tau}(t-s)}\mathrm{d} s\\
\leq & ce^{\mu \tau} e^{\mu t}r(\omega)+e^{\mu \tau}\left\|\psi\left(\theta_{-t} \omega\right)\right\|_{co}^{\mathcal{C}}+\varepsilon e^{\mu\tau}L_{f}\left\|\psi\left(\theta_{-t} \omega\right)\right\|_{co}^{\mathcal{C}} \int_{0}^{t}e^{\varepsilon L_{f}e^{\mu \tau}(t-s)}\mathrm{d} s\\
&+c\varepsilon e^{\mu\tau}L_{f}r(\omega)\int_{0}^{t}e^{\varepsilon L_{f}e^{\mu \tau}(t-s)}e^{\mu s}\mathrm{d} s.
\end{aligned}
\end{equation}
Therefore, we have
\begin{equation}\label{4.8}
\begin{aligned}
\left\|v_{t}\left(\cdot, \theta_{-t} \omega, \psi\left(\theta_{-t} \omega\right)\right)\right\|_{co}^{\mathcal{C}} \leq & ce^{\mu \tau}r(\omega)+e^{-\mu t}e^{\mu \tau}\left\|\psi\left(\theta_{-t} \omega\right)\right\|_{co}^{\mathcal{C}}+\left\|\psi\left(\theta_{-t} \omega\right)\right\|_{co}^{\mathcal{C}}(e^{(\varepsilon L_{f}e^{\mu \tau}-\mu)t}-1)\\  &+\frac{c\varepsilon e^{\mu\tau}L_{f}}{\mu-\varepsilon e^{\mu\tau}L_{f}}[e^{-\varepsilon e^{\mu\tau}L_{f}}-e^{-\mu t}]r(\omega).
\end{aligned}
\end{equation}
Note that $\psi(\omega)=\phi-z(\theta_{t+\cdot}\omega).$ The above estimate \eqref{4.8}  implies that, for all $t \geq \tau,$
\begin{equation}\label{4.11}
\begin{aligned}
\left\|u_{t}\left(\cdot, \theta_{-t} \omega, \phi\right)\right\|_{co}^{\mathcal{C}} \leq & \left\|v_{t}\left(\cdot, \theta_{-t} \omega, \psi\left(\theta_{-t} \omega\right)\right)\right\|_{co}^{\mathcal{C}}+\|z(\theta_{-t}\theta_{t+\cdot}\omega)\|_{co}^{\mathcal{C}} \\
\leq & ce^{\mu \tau}r(\omega)+e^{-\mu t}e^{\mu \tau}\left\|\psi\left(\theta_{-t} \omega\right)\right\|_{co}^{\mathcal{C}}+\left\|\psi\left(\theta_{-t} \omega\right)\right\|_{co}^{\mathcal{C}}(e^{(\varepsilon L_{f}e^{\mu \tau}-\mu)t}-1)\\  &+\frac{c\varepsilon e^{\mu\tau}L_{f}}{\mu-\varepsilon e^{\mu\tau}L_{f}}[e^{-\varepsilon e^{\mu\tau}L_{f}}-e^{-\mu t}]r(\omega)+ce^{\frac{\mu \tau}{2}}r(\omega).
\end{aligned}
\end{equation}
Therefore, if $\phi \in \mathcal{D}\left(\theta_{-t} \omega\right)$ and $\varepsilon L_{f}e^{\mu \tau}-\mu<0$, then there exists a $T_{\mathcal{D}}>\tau$ such that, for all $t \geq T_{D}(\omega)$,
\begin{equation}\label{4.12}
\begin{aligned}
e^{-\mu t}e^{\mu \tau}\left\|\psi\left(\theta_{-t} \omega\right)\right\|_{co}^{\mathcal{C}}+\left\|\psi\left(\theta_{-t} \omega\right)\right\|_{co}^{\mathcal{C}}(e^{(\varepsilon L_{f}e^{\mu \tau}-\mu)t}-1)-\frac{c\varepsilon e^{\mu\tau}L_{f}}{\mu-\varepsilon e^{\mu\tau}L_{f}}e^{-\mu t}r(\omega) \leq c_1(\omega),
\end{aligned}
\end{equation}
which, along with \eqref{4.11} shows that, for all $t \geq T_{\mathcal{D}}(\omega)$
\begin{equation}\label{4.13}
\begin{aligned}
\left\|u_{t}\left(\cdot, \theta_{-t} \omega, \phi\right)\right\|_{co}^{\mathcal{C}} \leq  2ce^{\mu \tau}r(\omega)+\frac{c\varepsilon e^{\mu\tau}L_{f}}{\mu-\varepsilon e^{\mu\tau}L_{f}}e^{-\varepsilon e^{\mu\tau}L_{f}}r(\omega)+c_1(\omega).
\end{aligned}
\end{equation}
Given $\omega \in \Omega,$ define
\begin{equation}\label{4.14}
\begin{aligned}
K(\omega)=\{\varphi \in \mathcal{C}:\|\varphi\|_{co}^{\mathcal{C}} \leq 2ce^{\mu \tau}r(\omega)+\frac{c\varepsilon e^{\mu\tau}L_{f}}{\mu-\varepsilon e^{\mu\tau}L_{f}}e^{-\varepsilon e^{\mu\tau}L_{f}}r(\omega)+c_1(\omega)\}.
\end{aligned}
\end{equation}
Then, $K=\{K(\omega)\}_{\omega \in \Omega} \in \mathcal{D}$. Furthermore, \eqref{4.13} implies that $K$ is a random absorbing set for the RDS $\Phi$ in $\mathcal{D}$.
\end{proof}

In the sequel, we first show that $\Psi$ is a continuous random semiflow with respect to $\|\cdot \|^{\mathcal{C}}_{co}$.

\begin{lem}\label{lem4.2}
Assume  (\textbf{H})  is satisfied and $f$ is globally bounded, then  $\Phi\left(t, \theta_{-t} \omega, \phi\right)$ is a continuous random semiflow  with respect to the compact open topology induced by the norm $\|\cdot\|^{\mathcal{C}}_{co}$.
\end{lem}
\begin{proof}
It suffices to prove that for any given $\phi_n \in D\left(\theta_{-t_{n}} \omega\right)$ and $\phi \in D\left(\theta_{-t} \omega\right)$ such that $t_n\rightarrow t, \phi_n\rightarrow \phi$ then the sequence $\Psi\left(t_{n}, \theta_{-t_{n}} \omega, \phi_n\right)$ convergent to $\Psi\left(t, \theta_{-t} \omega, \phi\right)$ with respect to $\|\cdot\|^{\mathcal{C}}_{co}$. Here, for convenience, we assume that $t_n\geq t$ since the case $t_n< t$ can be proved similarly.

To prove the continuity, define  $P: \mathbb{R}_{+} \times\Omega\times \mathcal{C} \rightarrow X_{co}$ by $P(t, \theta_{-t} \omega, \phi)(x)=\Psi(t, \theta_{-t} \omega, \phi)(0)(x)$ for all $(t, \phi) \in \mathbb{R}_{+} \times \mathcal{C}$. By Theorem \ref{thm3.2} and the cocycle property of $z$, we only need to prove  that $P(t, \theta_{-t} \omega, \phi)(x)$ is a random semiflow . Take $\left\{\left(t_{n}, \phi_{n}\right)\right\}_{n \in \mathbb{N}} \subset \mathbb{R}_{+} \times D\left(\theta_{-t_{n}} \omega\right)$  such that $\lim _{n \rightarrow \infty}\left|t_{n}-t\right|=0$ and
$\lim _{n \rightarrow \infty} d_{r}\left(\phi_{n}, \phi\right)=0$. Denote by $G_n\triangleq |P(t_n, \theta_{-t_{n}} \omega, \phi_n)(x)-P(t, \theta_{-t} \omega, \phi)(x)|$. For any given a bounded closed interval $I$ and any $x\in I$,  we have
\begin{equation}\label{5.1}
\begin{aligned}
G_n&=\left|\Psi(t_n, \theta_{-t_{n}} \omega, \phi_n)(0)(x)-\Psi\left(t, \theta_{-t} \omega, \phi\right)(0)(x)\right|\\
&\leq\left|\Phi(t_n, \theta_{-t_{n}} \omega, \psi_n)(0)(x)-\Phi\left(t, \theta_{-t} \omega, \psi\right)(0)(x)\right|+|z(\theta_t\omega)(x)-z(\theta_{t_n}\omega)(x)|\ \\
&\leq|S(t_n) \psi_n(0,x)-S(t)\psi(0,x)|+|\int_{0}^{t_n} S(t_n-r)F\left(v^n_{r}+z\left(\theta_{-t_n}\theta_{r+\cdot} \omega\right)\right)(x)-\\
&\int_{0}^{t}S(t-r) F\left(v_{r}+z\left(\theta_{-t}\theta_{r+\cdot} \omega\right)\right)(x) dr|+|\int_{0}^{t_n} S(t_n-r) \Delta z\left(\theta_{-t_n}\theta_{r} \omega\right)(x) dr-\\
&\int_{0}^{t} S(t-r) \Delta z\left(\theta_{-t}\theta_{r} \omega\right)(x) dr|+|z(\theta_t\omega)(x)-z(\theta_{t_n}\omega)(x)| \\
& \triangleq I_1+I_2+I_3+I_4,
\end{aligned}
\end{equation}
for  $P$-a.e. $\omega\in \Omega$.
Now, we estimate each term on the right handside of  \eqref{5.1}.
\begin{equation}\label{5.2}
\begin{aligned}
I_1&\leqq |[S(t_n)-S(t)] \psi_n(0,x)|+|S(t)[\psi_n(0,x)-\psi(0,x)]|\\
&\leq\frac{\left(1+t_{n}\right) \exp \left(-\mu t_{n}\right)|\psi_n(0,x)|}{t_{n}}\left|t-t_{n}\right|+e^{-\mu t}|\psi_n(0,x)-\psi(0,x)|\\
&\triangleq I_{11}+I_{12}.
\end{aligned}
\end{equation}
It follows from $\lim _{n \rightarrow \infty} d_{r}\left(\phi_{n}, \phi\right)=0$, the continuity of $z(\theta_t\omega)$ with respect to $t$ and $\lim _{n \rightarrow \infty}\left|t_{n}-t\right|=0$ that $\lim _{n \rightarrow \infty}I_{12}=0$. Clearly, $I_{11}\rightarrow 0$ because of $t_n\rightarrow t$.
\begin{equation}\label{5.3}
\begin{aligned}
I_2&\leqq |\int_{0}^{t_n} S(t_n-r)F\left(v^n_{r}+z\left(\theta_{-t_n}\theta_{r+\cdot} \omega\right)\right)(x)dr-\int_{0}^{t} S(t_n-r)F\left(v^n_{r}+z\left(\theta_{-t_n}\theta_{r+\cdot} \omega\right)\right)(x)dr|+\\
& |\int_{0}^{t} S(t_n-r)F\left(v^n_{r}+z\left(\theta_{-t_n}\theta_{r+\cdot} \omega\right)\right)(x)dr-\int_{0}^{t} S(t_n-r)F\left(v_{r}+z\left(\theta_{-t_n}\theta_{r+\cdot} \omega\right)\right)(x)dr|+
\\&|\int_{0}^{t} S(t_n-r)F\left(v_{r}+z\left(\theta_{-t_n}\theta_{r+\cdot} \omega\right)\right)(x)dr-\int_{0}^{t}S(t-r) F\left(v_{r}+z\left(\theta_{-t_n}\theta_{r+\cdot} \omega\right)\right)(x) dr|+
\\&|\int_{0}^{t} S(t_n-r)F\left(v_{r}+z\left(\theta_{-t_n}\theta_{r+\cdot} \omega\right)\right)(x)dr-\int_{0}^{t}S(t-r) F\left(v_{r}+z\left(\theta_{-t}\theta_{r+\cdot} \omega\right)\right)(x) dr|\\
& \triangleq I_{21}+I_{22}+I_{23}+I_{24}.
\end{aligned}
\end{equation}
By Lemmas \ref{lem2.1} and \ref{lem3.2} and the boundedness of $f$, we have for any $x\in I$ and $P$-a.e. $\omega \in \Omega$
 \begin{equation}\label{5.4}
\begin{aligned}
I_{21}&\leq \varepsilon\int_{t}^{t_n}e^{-\mu (t_n-r)}f\left(v^n_{r}+z\left(\theta_{r+\cdot} \omega\right)\right)(x)dr\leq \varepsilon M \int_{t}^{t_n}e^{-\mu (t_n-r)}dr.
\end{aligned}
\end{equation}
Since $t_n\rightarrow t$, we can see $I_{21}\rightarrow 0$.
 \begin{equation}\label{5.5}
\begin{aligned}
I_{22}&\leq \varepsilon L_f \int_{0}^{t}e^{-\mu (t_n-r)}\|v^n_{r}-v_{r}\|_{\mathcal{C}_{co}}dr\leq \varepsilon L_f |v_{r}(\cdot,\theta_{-t_n}\omega, \psi_n)(x)-v_{r}(\cdot,\theta_{-t_n}\omega, \psi)(x)| \int_{t}^{t_n}e^{-\mu (t_n-r)}dr.
\end{aligned}
\end{equation}
By similar procedure as the proof of Theorem 2.8-(i) in \cite{YCWT}, we have  $I_{22}\rightarrow 0$ provided $\psi_n\rightarrow \psi$.
\begin{equation}\label{5.6}
\begin{aligned}
I_{23}&\leq \int_{0}^{t} [S(t_n-r)-S(t-r)]F\left(v_{r}+z\left(\theta_{r+\cdot} \omega\right)\right)dr \\
&= \int_{0}^{t_n-\sqrt{\delta}}[S(t_n-r)-S(t-r)]F\left(v_{r}+z\left(\theta_{r+\cdot} \omega\right)\right)dr+\int_{t_n-\sqrt{\delta}}^{t} [S(t_n-r)-S(t-r)]F\left(v_{r}+z\left(\theta_{r+\cdot} \omega\right)\right)dr\\
& \leq \varepsilon M  \int_{0}^{t_n-\sqrt{\delta}}\frac{\left(1+\mu (t_{n}-r)\right) \exp \left(-\mu(t_{n}-r)\right)}{t_{n}-r}(t_n-t)dr+2M|t-t_n+\sqrt{\delta}|\\
&\leq\varepsilon M\left|t-t_{n}\right|(\frac{1}{\sqrt{\delta}}+\mu)+2M|t-t_n+\sqrt{\delta}|,
\end{aligned}
\end{equation}
where $\delta\in (0,1)$ and $t<t_n<t+\delta$ with $\delta \rightarrow 0$. Hence, we have $I_{23}\rightarrow 0$ when $t_n\rightarrow t$. Therefore, $\lim _{n \rightarrow \infty}I_2=0$, as $t_n\rightarrow t, \phi_n\rightarrow \phi$.
In the following, we estimate $I_3$.
\begin{equation}\label{5.7}
\begin{aligned}
I_3&\leqq |\int_{0}^{t_n} S(t_n-r)\Delta z\left(\theta_{-t_n}\theta_{r} \omega\right)(x) -\int_{0}^{t}S(t-r) \Delta z\left(\theta_{-t_n}\theta_{r} \omega\right)(x)  dr|\\
&\leqq |\int_{0}^{t_n} S(t_n-r)\Delta z\left(\theta_{-t_n}\theta_{r} \omega\right)(x) dr-\int_{0}^{t} S(t_n-r)\Delta z\left(\theta_{-t_n}\theta_{r} \omega\right)(x) dr|+\\
& |\int_{0}^{t} S(t_n-r)\Delta z\left(\theta_{-t_n}\theta_{r} \omega\right)(x) dr-\int_{0}^{t} S(t-r)\Delta z\left(\theta_{-t_n}\theta_{r} \omega\right)(x) dr| +\\
& |\int_{0}^{t} S(t-r)\Delta z\left(\theta_{-t_n}\theta_{r} \omega\right)(x) dr-\int_{0}^{t} S(t-r)\Delta z\left(\theta_{-t}\theta_{r} \omega\right)(x) dr| \triangleq I_{31}+I_{32}+I_{33}.
\end{aligned}
\end{equation}
 \begin{equation}\label{5.8}
\begin{aligned}
I_{31}&\leq c |\Delta z\left(\theta_{r} \omega\right)(x)|\int_{t}^{t_n}e^{-\mu (t_n-r)} e^{\mu (t_n-r)/2}r(\omega)dr.
\end{aligned}
\end{equation}
Since $\{g_j(x)\}_{j=1}^{m}$ are twice continuously differentiable, there exists $M>0$ such that for any $x\in I$ and  $P$-a.e. $\omega$, $|\Delta z\left(\theta_{r} \omega\right)(x)|\leqslant M$. Thus, $\lim _{n \rightarrow \infty}I_{31}=0$ and $\lim _{n \rightarrow \infty}I_{31}=0$ in the case $t_n\rightarrow t$. It follows from $\Delta z\left(\theta_{t} \omega\right)$ is continuous with respect to $t$ that $\lim _{n \rightarrow \infty}I_{33}=0$. By the same arguments as the estimation of $I_{23}$ in \eqref{5.6}, we have $\lim _{n \rightarrow \infty}I_{32}=0$, indicating that $\lim _{n \rightarrow \infty}I_{3}=0$. Moreover, since   $z_{j}\left(\theta_{t} \omega_{j}\right)$ is $P$-a.e. $\omega$ continuous, we have that $\lim _{n \rightarrow \infty}I_4=0$.
Summing up the above computation together with  Lemma \ref{lem5.1}, we can see that  $\Phi\left(t, \theta_{-t} \omega, \phi\right)$ is a continuous random semiflow  with respect to the compact open topology induced by the norm $\|\cdot\|_{co}$.
\end{proof}

Now, we are in the position to prove the $\mathcal{D}$-pullback asymptotically compact in $\mathcal{D}$ with respect to $\|\cdot\|_{C}^{co}$.

\begin{lem}\label{lem4.4}
Assume that  $(\mathbf{H})$ holds and $f$ is bounded. Then, the RDS $\Psi$ generated by SNDRDE \eqref{4} is $\mathcal{D}$-pullback asymptotically compact in $\mathcal{C}_{co}$ for $t>\tau$, i.e., for $P$-a.e. $\omega \in \Omega$, the sequence $\{\Psi(t_{n}, \theta_{-t_{n}} \omega,\phi_n\left(\theta_{-t_{n}} \omega\right))\}$ has a convergent subsequence in $\mathcal{C}_{co}$ provided $t_{n} \rightarrow \infty$, $B=\{B(\omega)\}_{\omega \in \Omega} \in \mathcal{D}$ and $\phi_n\left(\theta_{-t_{n}} \omega\right) \in B\left(\theta_{-t_{n}} \omega\right)$
\end{lem}
\begin{proof}
Take an arbitrary random set $\{B(\omega)\}_{\omega \in \Omega} \in \mathcal{D}$, a sequence $t_{n} \rightarrow+\infty$ and $\phi_n \in B\left(\theta_{-t_{n}} \omega\right)$. We have to prove that $\left\{\Psi\left(t_{n}, \theta_{-t_{n}} \omega, \phi_n\right)\right\}$ is precompact. Since $\{K(\omega)\}$ is a random absorbing for $\Psi$, then there exists $T>0$ such that, for all $\omega \in \Omega$,
\begin{equation}\label{4.24}
\Psi\left(t, \theta_{-t} \omega\right) B\left(\theta_{-t} \omega\right) \subset K(\omega)
\end{equation}
for all $t \geq T$.
Because $t_{n} \rightarrow+\infty$, we can choose $n_{1} \geq 1$ such that $t_{n_{1}}-1 \geq T$. Applying \eqref{4.24} for $t=t_{n_{1}}-1$ and $\omega=\theta_{-1} \omega$, we find that
\begin{equation}\label{4.25}
\eta_{1} \triangleq \Psi\left(t_{n_{1}}-1, \theta_{-t_{n_{1}}} \omega, \phi_{n_{1}}\right)  \in K\left(\theta_{-1} \omega\right)
\end{equation}
Similarly, we can choose a subsequence $\left\{n_{k}\right\}$ of $\{n\}$ such that $n_{1}<n_{2}<\cdots<n_{k} \rightarrow$ $+\infty$ such that
\begin{equation}\label{4.26}
\eta_{k} \triangleq \Psi\left(t_{n_{k}}-k, \theta_{-t_{n_{k}}} \omega, \phi_{n_{k}}\right)  \in K\left(\theta_{-k} \omega\right)
\end{equation}
Hence, by the assumption we conclude that
the sequence
\begin{equation}\label{4.27}
\left\{\Psi\left(k, \theta_{-k} \omega, \eta_{k}\right)\right\}
\end{equation}is precompact.
On the other hand by \eqref{4.26}, we have
\begin{equation}\label{4.28}
\begin{aligned}
\Psi(k, \theta_{-k} \omega, \eta_{k}) &=\Psi(k, \theta_{-k} \omega,\Psi(t_{n_{k}}-k, \theta_{-t_{n_{k}}} \omega, \phi_{n_{k}}) =\Psi\left(t_{n_{k}}, \theta_{-t_{n_{k}}} \omega,\phi_{n_{k}}\right),
\end{aligned}
\end{equation}
for all $k \geq 1$. Combining \eqref{4.27} and \eqref{4.28}, we obtain that the sequence $\left\{\Psi\left(t_{n_{k}}, \theta_{-t_{n_{k}}} \omega,\phi_{n_{k}}\right)\right\}$ is precompact. Therefore,  $\left\{\Psi\left(t_{n}, \theta_{t_{n}} \omega,\phi_{n_{k}}\right) \right\}$ is precompact, which completes the proof.
\end{proof}

Lemma \eqref{lem4.1} says that the continuous RDS $\Psi$ has a random absorbing set while Lemma \eqref{lem5.1} tells us that $(\theta, \Psi)$ is pullback asymptotically compact in $C_{co}$. Thus, it follows from Lemma \ref{lem1} that the continuous RDS $(\theta, \Psi)$ possesses a random attractor. Namely, we obtain the following result.

\begin{thm}\label{thm5.1} Assume that $(\mathbf{H})$ holds, $\varepsilon L_{f}e^{\mu \tau}-\mu<0$ and $f$ is bounded, then the continuous RDS $\Psi$ generated by  \eqref{4} admits a unique $\mathcal{D}$-pullback attractor in $\mathcal{C}_{co}$ belonging to the class $\mathcal{D}$.
\end{thm}

\section{Existence of exponentially attracting stationary solutions}

In this section, we are devoted to  deriving sufficient conditions that guarantee  the random attractor being  an exponentially attracting random fixed point $\xi^{*}$ by adopting the general Banach fixed point theorem. We first introduce the general Banach fixed point theorem, which was established in \cite{26} and extended in \cite{DLS03} to  infinite case in the following.
\begin{lem}\label{lem6.1}Let $(Y, d_Y)$ be a complete metric space with bounded metric. Suppose that
$$
\Phi(t, \omega, Y) \subset Y
$$
for $\omega \in \Omega, t \geq 0$, and that $x \rightarrow \Phi(t, \omega, x)$ is continuous. In addition, we assume the contraction condition: There exists a constant $k<0$ such that, for $\omega \in \Omega$,
$$
\sup _{x \neq y \in Y} \log \frac{d_Y(\Phi(1, \omega, x), \Phi(1, \omega, y))}{d_Y(x, y)} \leq k
$$
Then $\Phi$ has a unique generalized fixed point $\gamma^{*}$ in $Y$. Moreover, the following convergence property holds:
$$
\lim _{t \rightarrow \infty} \Phi\left(t, \theta_{-t} \omega, x\right)=\gamma^{*}(\omega)
$$
for any $\omega \in \Omega$ and $x \in Y$.
\end{lem}

\begin{thm}\label{Theorem 6} Assume that $f$ is bounded and satisfies $\left(\mathbf{H}\right)$. Moreover, assume that $0<\tau<1$ and  $\mu>\max\{\frac{\varepsilon L_f}{1-\tau}, \varepsilon L_{f}e^{\mu \tau}\}$. Then the random dynamical system $\Psi$ generated by SNDRDE \eqref{4} possess a tempered random fixed point $\xi^{*}$, which is unique under all tempered random variables in $\mathcal{C}_{co}$  and  attracts exponentially fast every random variable in $\mathcal{C}_{co}$.
\end{thm}
\begin{proof}
If $\mu>\varepsilon L_{f}e^{\mu \tau}$, then the conditions of Theorem \ref{thm5.1} hold and hence \eqref{4} possess random attractors in $\mathcal{C}_{co}$. We will prove that  \eqref{4} admits a unique globally exponentially attracting random stationary solution in $\mathcal{C}_{co}$, which immediately implies the random attractor in $\mathcal{C}_{co}$ obtained in Theorem \ref{thm5.1} is the random fixed point. If suffices to prove that  the RDS $\Phi$  generated by \eqref{3.10} has a unique exponentially attracting generalized fixed point $\chi^{*}$. Since  the transformation $v(t)=u(t)-z\left(\theta_{t} \omega\right)$ and $u(t)=v(t)+z\left(\theta_{t} \omega\right)$ are  conjugation,  one can see that $\xi^{*}=\chi^{*}+z\left(\theta_{t} \omega\right)$ is a unique exponentially attracting generalized fixed point of \eqref{4} by conjugation technique.

By  \eqref{4.3} and the Gr\"{o}nwall inequality, one can see that for any $\psi \in \mathcal{C}_{co}$
\begin{equation}\label{4.8}
\begin{aligned}
\left\|v_{t}\left(\cdot,  \omega, \psi\left(\omega\right)\right)\right\|_{co}^{\mathcal{C}} \leq & ce^{\mu \tau}r(\omega)+e^{-\mu t}e^{\mu \tau}\left\|\psi\left( \omega\right)\right\|_{co}^{\mathcal{C}}+\left\|\psi\left( \omega\right)\right\|_{co}^{\mathcal{C}}(e^{(\varepsilon L_{f}e^{\mu \tau}-\mu)t}-1)\\  &+\frac{c\varepsilon e^{\mu\tau}L_{f}}{\mu-\varepsilon e^{\mu\tau}L_{f}}[e^{-\varepsilon e^{\mu\tau}L_{f}}-e^{-\mu t}]r(\omega),
\end{aligned}
\end{equation}
which implies that for any $\psi \in \mathcal{C}_{co}$, $\Phi(t,\omega,\psi )\in \mathcal{C}_{co}$, i.e. $\mathcal{C}_{co}$ is invariant under the random semiflow $\Phi$. Moreover, it follows from Lemma \ref{lem4.2} that $\Phi$ is continuous in $\mathcal{C}$.
Therefore, we only need to prove the contraction property. That is, there exists $k<0$ such that
 \begin{equation}\label{6.1}
\begin{aligned}
\sup _{\varphi \neq \psi \in \mathcal{C}_{co}}\|\Phi(1, \omega, \varphi)-\Phi(1, \omega, \psi)\|_{co}^{\mathcal{C}}\leq e^k\|\varphi-\psi\|_{co}^{\mathcal{C}}.
\end{aligned}
\end{equation}
Hence, it suffices to prove that for any $\varphi,\psi\in \mathcal{C}$
 \begin{equation}\label{6.2}
\begin{aligned}
\|\Phi(1, \omega, \varphi)-\Phi(1, \omega, \psi)\|_{co}^{\mathcal{C}}=\|v_{1}(\cdot, \omega,\varphi)-v_{1}(\cdot, \omega,\psi)\|_{co}^{\mathcal{C}}\leq e^k\|\varphi(\zeta,x)-\psi(\zeta,x)\|_{co}^{\mathcal{C}}.
\end{aligned}
\end{equation}
By Eq. \eqref{3.13}, we have for any $\varphi,\psi\in \mathcal{C}$
 \begin{equation}\label{6.3}
\begin{aligned}
\|v_{1}(\cdot, \omega,\varphi)-v_{1}(\cdot, \omega,\psi)\|_{co}^{\mathcal{C}}& \leq \|S(1)[\varphi(0)- \psi(0)]\|+\\
&\sup _{\zeta\in[-\tau,0]} \int_{0}^{1+\zeta} S(1+\zeta-r)[F\left(v^\varphi_{r}+z\left(\theta_{r+\cdot} \omega\right)\right)- F\left(v^\phi_{r}+z\left(\theta_{r+\cdot} \omega\right)\right)] d r\\
& \leq  e^{\mu \tau}[e^{-\mu}\|\phi-\psi\|_{co}^{\mathcal{C}}+  \varepsilon L_{f}\int_{0}^{1} e^{-\mu (1-r)}\left\|v_{r}(\cdot, \omega,\phi)-v_{r}(\cdot, \omega,\psi)\right\|_{co}^{\mathcal{C}} \mathrm{d}r].
\end{aligned}
\end{equation}
Multiply both sides of \eqref{6.3} by $e^\mu$ leads to
 \begin{equation}\label{6.3b}
\begin{aligned}
e^u\|v_{1}(\cdot, \omega,\varphi)-v_{1}(\cdot, \omega,\psi)\|_{co}^{\mathcal{C}}
& \leq  e^{\mu \tau}[\|\varphi-\psi\|_{co}^{\mathcal{C}}+  \varepsilon L_{f}\int_{0}^{1} e^{\mu r}\left\|v_{r}(\cdot, \omega,\varphi)-v_{r}(\cdot, \omega,\psi)\right\|_{co}^{\mathcal{C}} \mathrm{d}r].
\end{aligned}
\end{equation}
Again, the Gr\"{o}nwall inequality gives rise to
 \begin{equation}\label{6.4}
\begin{aligned}
e^u\|v_{1}(\cdot, \omega,\varphi)-v_{1}(\cdot, \omega,\psi)\|_{co}^{\mathcal{C}}& \leq   e^{\mu \tau+\varepsilon L_{f}}\|\varphi-\psi\|_{co}^{\mathcal{C}},
\end{aligned}
\end{equation}
implying that $\|v_\varphi-v_\psi\|_{co}^{\mathcal{C}}\leq e^{\mu (\tau-1)+\varepsilon L_{f}}\|\phi-\psi\|_{co}^{\mathcal{C}}$. Thus, $\mu> \frac{\varepsilon L_f}{1-\tau}$ means that $\mu (\tau-1)+\varepsilon L_{f}<0$, indicating that \eqref{6.1} satisfies. Therefore,  $\Phi(t,\omega, \cdot)$ admits a random exponentially attracting generalized fixed point $\chi^{*}$ and $\xi^{*}=\chi^{*}+z\left(\theta_{t} \omega\right)$ is a unique exponentially attracting generalized fixed point of \eqref{4}.
This completes the proof.

\end{proof}

\section{Summary}
In this paper, we have obtained the existence and qualitative property of random attractors for \eqref{4} on a the simi-infinite interval $\mathbb{R}_+$. We show that under certain conditions, the random attractor is a  globally exponentially attracting random stationary solution. From dynamical system theory, the conditions for the attractors being fixed point are so strong that could be hardly met in the real world applications. Indeed, from the dissipative system theory, if some estimates on the dimension of random attractors can be given, it will benefit the researchers a lot in studying the structure of the random attractor. Nevertheless, the lack of inner product of the phase space and the asymmetry as well as noncompactness of  spatial domain causing this problem quite challenging and deserving further studies. Furthermore, in order to obtain the  global complex dynamics and nonlocal analysis of the qualitative properties of the system, existence and structure  of the associated invariant manifolds of the stationary solutions, the existence of connecting orbits (including the heteroclinic orbits or homoclinic orbits) are all of great significance and deserve  much attention.

\section{Acknowledgement}
This work was jointly supported by the National Natural Science Foundation of China (62173139), China Postdoctoral Science Foundation (2019TQ0089), Hunan Provincial Natural Science Foundation of China (2020JJ5344) the Science and Technology Innovation Program of Hunan Province (2021RC4030), the Scientific Research Fund of Hunan Provincial Education Department (20B353). \\
The research of T. Caraballo has been partially supported by Spanish Ministerio de Ciencia e
Innovaci\'{o}n (MCI), Agencia Estatal de Investigaci\'{o}n (AEI), Fondo Europeo de
Desarrollo Regional (FEDER) under the project PID2021-122991NB-C21 and the Junta de Andaluc\'{i}a
and FEDER under the project P18-FR-4509.

\end{document}